\documentclass[a4paper,12pt]{amsart}
\usepackage{amssymb, amsmath, ascmac, amsthm, geometry}

\usepackage[bookmarkstype=toc, colorlinks=true, linkcolor=blue, pdfborder={0 0 0}, bookmarks=true, bookmarksnumbered=true]{hyperref}
\newtheorem{introtheorem}{Theorem}
\newtheorem{introprop}[introtheorem]{Proposition}
\newtheorem{theorem}{Theorem}
\newtheorem{lemma}[theorem]{Lemma}

\newtheorem{prop}[theorem]{Proposition}
\theoremstyle{definition}
\newtheorem{definition}[theorem]{Definition}
\theoremstyle{remark}

\newtheorem{remark}[theorem]{Remark}

\def\dets{(\det{S})}
\def\no{\nonumber}
\def\pt#1#2{\frac{\partial #1}{\partial #2}}
\def\dpt#1#2{\dfrac{\partial #1}{\partial #2}}

\def\R{\mathbb{R}}

\def\O{\mathbb{O}}

\def\w{\wedge}
\def\we{\wedge}
\def\cal#1{\mathcal{#1}}
\def\rm#1{\mathrm{#1}}
\def\bf#1{\mathbf{#1}}
\def\un#1{\underline{#1}}
\def\vol{\rm{vol}}
\def\tr{\rm{tr}}
\def\ric{(\rm{Ric})}
\def\a{\alpha}
\def\al{\alpha}
\def\b{\beta}

\def\g{\gamma}
\def\ga{\gamma}

\def\d{\delta}

\def\e{\epsilon}

\def\k{\kappa}

\def\r{\rho}
\def\s{\sigma}
\def\o{\omega}
\def\om{\omega}
\def\Om{\Omega}

\def\gl#1{\mathrm{GL}(#1;\R)}

\def\sym#1{\mathrm{Sym}(#1;\mathbb{R})}
\def\psym#1{\mathrm{Sym}^+(#1;\mathbb{R})}
\def\so#1{\mathrm{SO}(#1)}
\def\fso#1{\mathfrak{so}(#1)}
\def\su#1{\mathrm{SU}(#1)}

\def\t#1{\mathrm{T}^{#1}}

\def\ant#1{\mathrm{Ant}(#1;\R)}

\begin{document}
\title{$G_2$-manifolds and the ADM formalism}
\author{Ryohei Chihara}
\begin{abstract}
In this paper we study a Hamiltonian function on the cotangent bundle of the space of Riemannian metrics on a 3-manifold $M$ and prove the orbits of the constrained Hamiltonian dynamical system correspond to $G_2$-manifolds foliated by hypersurfaces diffeomorphic to $M\times \mathrm{SO}(3)$. 
\end{abstract}
\address{Graduate School of Mathematical Sciences, The University of Tokyo, 3-8-1 Komaba Meguro-ku Tokyo 153-8914, Japan}
\email{rchihara@ms.u-tokyo.ac.jp}
\maketitle
\section*{Introduction}\label{sec:int}

Let $M$ be a closed oriented 3-manifold. Denote by $\bf{M}$ the space of Riemannian metrics on $M$ and by $T^*\bf{M}$ the cotangent bundle of $\bf{M}$. The cotangent bundle $T^*\bf{M}$ has the standard symplectic structure as in cases of finite-dimensional ones. The symplectic geometry on $T^*\bf{M}$ has been studied in relation with the initial value problem for the vacuum Einstein equations. In \cite{ADM}, the vacuum Einstein equations were reformulated as constrained Hamiltonian dynamical systems on $T^*\bf{M}$. The Hamiltonian formalism is called the ADM formalism, named after the authors Arnowitt, Deser and Misner. By choosing specific lapse functions and shift vector fields, the Hamiltonian function is defined by
\begin{align*}
 H_{GR}(\g,\pi)=-\int_{M}\mathrm{R}(\ga)\mathrm{vol}(\g)+\int_{M}\left(\rm{tr}(\pi^2)-\frac{1}{2}\rm{tr}(\pi)^2\right)\mathrm{vol}(\g)^{-1}\quad\text{on}\quad T^*\bf{M}.
\end{align*}
Here $(\g,\pi)\in T^*\bf{M}$, $\g\in \bf{M}$ and $\pi$ is a symmetric $(2,0)$-tensor field tensored with a volume form on $M$.
Also $\rm{R}(\g)$, $\rm{vol}(\g)$, $\rm{tr}(\pi)$ and $\rm{tr}(\pi^2)$ denote the scalar curvature, volume element of $\g$ and traces of $\pi$ and $\pi^2$, respectively. Then for a curve $(\g(t),\pi(t))$ in $T^*\bf{M}$ defined on an interval $(t_1,t_2)$, they 
showed that the Lorentzian metric $-(dt)^2+\g_{ij}dx^idx^j$ on $(t_1,t_2)\times M$ was Ricci-flat (i.e., a vacuum solution) if and only if $(\g(t),\pi(t))$ was an orbit of the Hamiltonian dynamical system of $H_{GR}$ satisfying the constraint conditions. The pointwise constraint conditions are
$R(\g)\rm{vol}(\g)^2+\dfrac{1}{2}\rm{tr}(\pi)^2-\rm{tr}(\pi^2)=0$ and $\displaystyle\sum_{j=1}^3\nabla_{j}\pi^{ij}=0$ for $i=1,2,3$. Here $\nabla$ denotes the Levi-Civita connection for each $\g(t)$. See \cite{ADM, BM, HE} for more details.

In the present paper, we study a Hamiltonian function similar to that of the ADM formalism defined by
\begin{align*}
H_{G_2}(\ga,\pi)=-\int_{M}\mathrm{R}(\ga)\mathrm{vol}(\ga) + \int_{M}\mathrm{det}(\pi)\mathrm{vol}(\ga)^{-2} \quad\text{on}\quad T^*\bf{M}.
\end{align*}
Here $\det{(\pi)}$ denotes the determinant of $\pi$, proportional to the third tensor power of $\rm{vol}(\g)$. 
As a main result, we establish a correspondence between the orbits of the constrained Hamiltonian dynamical system of $H_{G_2}$ and a class of $\so3$-invariant $G_2$-manifolds foliated by hypersurfaces diffeomorphic to $M\times \so3$. 
In the following, we explain how this ADM-like formulation of $G_2$-manifolds is obtained, and state our main results more precisely.

Here we recall the notion of a $G_2$-manifold. A $G_2$-manifold is a 7-dimensional Riemannian manifold with holonomy group contained in the exceptional Lie group $G_2$. 
Such a manifold is Ricci-flat, and one with full holonomy $G_2$ is an example of Berger's list of Riemannian manifolds with special holonomy. A $G_2$-manifold is determined by a torsion-free $G_2$-structure. Moreover a $G_2$-structure on a 7-dimensional manifold is specified by a 3-form pointwisely isomorphic to $\langle x\times y,z\rangle$ for $x,y,z\in\R^7.$ Here $\langle *,*\rangle$ and $\times$ denote the scalar product and cross product on $\R^7$, defined by the structure of Octonions $\O$. See \cite{BS, CS, Joy} for more details.

In \cite{Hit}, in terms of differential forms, Hitchin described locally $G_2$-manifolds as orbits of a constrained Hamiltonian dynamical system on the direct product of a fixed cohomology class of 3- and 4-forms on a 6-manifold. Then each point of orbits corresponds to an $\su3$-structure on the 6-manifold. This formulation is similar to the ADM formalism above in that the orbits of both dynamical systems generate equidistant hypersurfaces in ambient spaces \cite{Bry}. But Hitchin's Hamiltonian function is a volume functional that contains no derivatives, and the relation between Hitchin's and the ADM formalism is somewhat obscure.

On the other hand, Bryant and Salamon constructed first examples of complete full-holonomy $G_2$-manifolds \cite{BS}. Ones of their examples are diffeomorphic to $M\times \R^4$ and naturally regarded as one-parameter families of $\su2$-invariant (almost) special Lagrangian fibrations $M\times S^3$. 

In \cite{Chi}, we generalized settings in \cite{BS} by applying Hitchin's description above of $G_2$-manifolds to $\so3$-invariant $\su3$-structures on $M\times \so3$ whose fibers $\so3$ are (almost) special Lagrangin submanifolds. Here we can identify $M\times\su2$ with $M\times\so3$ by the double covering, since our results in \cite{Chi} are restricted to the description of regular parts. We showed one-to-one correspondence such $\su3$-structures and triples of a solder 1-forms, connection 1-forms and equivariant positive-definite symmetric $3\times3$ matrix-valued functions on $M\times\so3$. Then we described locally a class of $\so3$-invariant $G_2$-manifolds by a constrained dynamical system on the space of triples on $M\times \so3$. For $\t3$-invariant $G_2$-manifolds, see \cite{MS}, in which the authors studied $\t3$-invariant ones in terms of toric geometry.

The present paper is regarded as a continuation of our previous work \cite{Chi}. Our main result gives a characterization of the dynamical system in \cite{Chi} as the constraint Hamiltonian dynamical system of $H_{G_2}$ on $T^*\bf{M}$. In this characterization, we use a principal bundle $\cal{P}\to T^*\bf{M}$, where $\cal{P}$ is isomorphic to the direct product of the space of solder 1-forms and the space of equivariant symmetric $3\times3$ matrix-valued functions on $M\times\so3$. Also the structure group is the gauge group of $M\times\so3$. The principal bundle $\cal{P}$ has a natural connection, and by this a curve $\un{c}(t)$ in $T^*\bf{M}$ defined on an interval $(t_1,t_2)$ is lifted to a curve $c(t)$ in $\cal{P}$. Then the lift $c(t)$ gives a $G_2$-structure on $M\times\so3\times(t_1,t_2)$. Our main results are as follows:
\begin{introtheorem}\label{a}
Let $\un{c}(t)=(\g(t),\pi(t))$ be a curve in $T^*\bf{M}$. Then a lift $c(t)$ of $\un{c}(t)$ gives a torsion-free $G_2$-structure if and only if the curve $\un{c}(t)$ satisfies the following:
\begin{itemize}
\item $\un{c}(t)$ is a solution of the Hamiltonian dynamical system of $H_{G_2}$.
\item $\pi/\rm{vol}(\g)$ is positive-definite and $\displaystyle\sum_{j=1}^3\nabla_{j}\pi^{ij}=0$ holds for $i=1,2,3$.
\end{itemize}
Here the constraint conditions above are all pointwise, and satisfied for all $t$.
\end{introtheorem}
Theorem \ref{a} is deduced from Theorem \ref{thm:con1} by Remark \ref{rem:prop1}. 

Moreover the orbits of the Hamiltonian dynamical system of $H_{G_2}$ preserve the divergence-free conditions for $\pi$ as in the case of that of the ADM formalism. 
\begin{introprop}\label{b}
Let $\un{c}(t)$ be an orbit of the Hamiltonian dynamical system of $H_{G_2}$. Assume that  $\displaystyle\sum_{j=1}^3\nabla_{j}\pi^{ij}(t_0)=0$ holds pointwisely for $i=1,2,3$ at some $t_0$. Then $\displaystyle\sum_{j=1}^3\nabla_{j}\pi^{ij}(t)=0$ holds for $i=1,2,3$ on the whole orbit.
\end{introprop}
Proposition \ref{b} is deduced from  Proposition \ref{prop:con2} by Remark \ref{rem:prop1}.

The present paper is organized as follows. In Section \ref{sec:pre}, we fix our notation and introduce the principal bundle $\cal{P}\to T^*\bf{M}$. Also we review some results in \cite{Chi}. Section \ref{sec:con} states our main results: Theorem \ref{thm:con1} and Proposition \ref{prop:con2} (corresponding to Thearem \ref{a} and Proposition \ref{b}). In Section \ref{sec:var}, we prepare some elementary variation formulas used in the proof of the main results. In Section \ref{sec:pro}, we prove the main results. Finally in Section \ref{sec:prop}, we study some properties of the Hamiltonian dynamical system of $H_{G_2}$: the behavior of the dynamical system for scaling of $H_{G_2}$ and variations of some functionals along the orbits. Moreover we observe Bryant-Salamon's examples in (\cite{BS}, Section 3) in our formulation, and remark briefly smooth continuation of the orbits to regions corresponding to the indefinite counterparts of $G_2$-manifolds.
\section{Preliminaries}\label{sec:pre}
\subsection{The phase space}
Let $\otimes^p E$, $S^pE$ and $\we^p E$ be the $p$-th tensor, symmetric, and anti-symmetric power of a vector bundle $E$ over a manifold $N$. Denote by $\Om^p(N, E)$ the space of $E$-valued $p$-forms on $N$. We treat all in the class of $C^{\infty}$, and omit the symbol of summation adopting the Einstein notation.

Let $M$ be a closed oriented $3$-manifold. The space of  Riemannian metrics on $M$ is denoted by
\begin{align*}
\mathbf{M}=\Om^0(M,S_{+}^2T^*M),
\end{align*}
where  $S_{+}^2T^*M$ is the positive-definite part of $S^2T^*M$. For $\ga \in \mathbf{M}$, the tangent space $T_{\ga}\mathbf{M}=\Om^0(M,S^2T^*M)$. Then $T\mathbf{M}=\mathbf{M}\times \Om^0(M,S^2T^*M)$.
The cotangent space $T^*_{\ga}\mathbf{M}=\Om^3(M,S^2TM)$ via $T^*_{\ga}\mathbf{M}\otimes T_{\ga}\mathbf{M} \to \mathbb{R}$ given by 
\begin{align}
\langle \pi,h \rangle = \int_{M}\pi^{ij}h_{ij} \quad\text{for}\quad  \pi \in \Om^3(M,S^2TM) \text{ and }h\in \Om^0(M,S^2T^*M).\no
\end{align}
Then we have
\begin{align}\no
T^*\mathbf{M}=\mathbf{M}\times \Om^3(M,S^2TM). 
\end{align}
The space $T^*\mathbf{M}$ has the standard symplectic form $\Om$ given by 
\begin{align}
\Om_{(\ga,\pi)}((h_1,\om_1),(h_2,\om_2))=\int_{M}(h_1\om_2-h_2\om_1) \quad\text{for}\quad(h_1,\om_1), (h_2,\om_2)\in T_{(\ga,\pi)}T^*\mathbf{M}.\no
\end{align}
 We can check $\Om=-d\theta$ for the tautological Liouville 1-form $\theta$ on $T^*\mathbf{M}$ as in cases of finite-dimensional cotangent bundles.

Let $P=M\times\so3$. Since every orientable 3-manifold $M$ is parallelizable, $P$ is regarded as a fixed oriented orthonormal frame bundle over $M$. 
Denote by $\Om^p(P;V)$ the space of $V$-valued $p$-forms on $P$ for a vector space $V$. For a representation $\rho: \mathrm{SO}(3) \to GL(V)$, denote by $\Om^p(P;V)^{\so3}_{hor}$ the space of $\al \in \Om^p(P;V)$ satisfying $R^*_{g}\al = \rho(g^{-1})\al$ and $\iota(A^*)\al=0$ for all $g\in \mathrm{SO}(3)$ and $A\in \mathfrak{so}(3)$. Here $R_g:P\to P$ denotes the right action, $\iota$ the inner product, and $A^*$ the infinitesimal vector field of $A$ on $P$. If $p=0$, then $\Om^0(P;V)^{\so3}_{hor}=\Om^0(P;V)^{\so3}$,  the space of $\so3$-equivariant $V$-valued functions on $P$. Take $\mathbb{R}^3$ with the usual representation of $\mathrm{SO}(3)$.
A 1-form $e=(e^1,e^2,e^3)\in \Om^1(P;\mathbb{R}^3)^{\so3}_{hor}$ is called a {\it solder 1-form} if $e_u^1\we e_u^2\we e_u^3 \neq 0$ for all $u \in P$. 
Put $\mathrm{vol}(e)=e^1\w e^2 \w e^3$, which is a volume form on $M$ because of the $\so3$-equivariance of $e$. Without loss of generality, we assume the orientation of $M$ coincides with $\mathrm{vol}(e)$.
Let $\mathcal{M}$ be the space of solder 1-forms on $P$. Denote by $\mathcal{G}$ the gauge group of $P$. An element $\tau\in \mathcal{G}$ acts on $e\in \mathcal{M}$ from right by $e\cdot\tau=\tau^*e$. For $e \in \mathcal{M}$, a Riemannian metric $\g$ is uniquely defined on $M$ such that $(s^*e^1,s^*e^2,s^*e^3)$ is a local orthonormal coframe for all local sections $s:U\to P$ on all domains $U$ of $M$. Thereby this map defines a natural principal $\mathcal{G}$-bundle over $\mathbf{M}$
\begin{align}\no
i:\mathcal{M} \to \mathbf{M},\quad i(e)=\g.
\end{align}
Let $\mathrm{M}(n;\R)$, $\sym{n}$ and $\ant{n}$ be the space of real $n\times n$ matrices, symmetric and anti-symmetric ones respectively, on which $g\in \mathrm{SO}(n)$ acts by $g\cdot B=gBg^{-1}$, where $B$ is an element of the spaces. For $e\in\mathcal{M}$, the tangent space is
\begin{align*}
    T_e\mathcal{M} &={} \Om^0(P;\mathrm{M}(3;\R))^{\so3} \\
    &={}\Om^0(P;\sym3)^{\so3}\oplus \Om^0(P;\ant3)^{\so3}
\end{align*}
by the identification
\begin{align*}
T_{e}\mathcal{M} \ni (\dot{e}^1,\dot{e}^2,\dot{e}^3)=(T_{1j}e^j,T_{2j}e^j,T_{3j}e^j)\mapsto (T_{ij}) \in \Om^0(P;\mathrm{M}(3;\R))^{\so3}.
\end{align*}
Here the $\so3$-equivariance of $e$ and $\dot{e}$ induces that of $T=(T_{ij})$. The infinitesimal action of $\mathcal{G}$ at  $e \in \mathcal{M}$ maps onto $\Om^0(P;\ant3)^{\so3} \subset T_{e}\cal{M}$. Hence the horizontal distribution
\begin{align*}
    H_{e}\mathcal{M}:=\Om^0(P;\sym3)^{\so3}\quad\text{for all }e\in \mathcal{M}
\end{align*}
gives a connection on $i:\mathcal{M}\to\mathbf{M}$.

\begin{remark}
The connection above is not flat. To see this, let us see the local model of $i:\cal{M} \to \bf{M}$. Let $\rm{Fr}^+(\mathbb{V})$ and $\rm{S}^2_{+}(\mathbb{V}^*)$ be the spaces of oriented basis and positive-definite symmetric 2-forms on a real $n$-dimensional vector space $\mathbb{V}$. Then we have a principal $\so{n}$-bundle $\rm{Fr}^+(\mathbb{V}) \to \rm{S}^2_{+}(\mathbb{V}^*)$. If we fix a basis on $\mathbb{V}$, then this bundle is identified with $\rm{GL}^+(n;\R) \to \psym{n}$. Here $\rm{GL}^+(n;\R)$ and $\psym{n}$ are the identity component of a general linear group and the positive-definite subset of $\sym{n}$. The decomposition $\rm{M}(n;\R)=\sym{n}\oplus \ant{n}$ gives a non-flat connection on the bundle since $[S_1,S_2] \in \ant{n}$ for $S_1,S_2\in \sym{n}$.
\end{remark}

The differential of  $i:\cal{M}\to\mathbf{M}$ restricts to the following principal $\cal{G}$-bundle over $T\mathbf{M}$
\begin{align*}
    i_{*}: H\cal{M}\to T\mathbf{M}, \quad i_{*}(e,T)=(\g,\rho),
\end{align*}
where $H\cal{M}:=\cal{M}\times \Om^{0}(P;\sym3)^{\so3}$ and $T\mathbf{M}=\mathbf{M}\times\Om^0(M,S^2T^*M)$. Here $\tau\in \cal{G}$ acts on $H\cal{M}$ by $(e,T)\cdot\tau=(\tau^*e,\tau^*T)$. The dual of $i_*$ defines a principal $\cal{G}$-bundle over $T^*\bf{M}$
\begin{align}\label{eq:pre3}
    j:\cal{M}\times \Om^3(P;\sym3)^{\so3}_{hor} \to T^*\bf{M},\quad j(e,U)=(i(e),\pi)
\end{align}
by
\begin{align}\no
   \langle \pi,\rho\rangle =\int_{M}\mathrm{tr}(TU)=\int_{M}T_{ij}U_{ij}
\end{align}
for $(i(e),\rho)=i_{*}(e,T)\in T_{\g}\bf{M}$ with $(e,T)\in H_{e}\cal{M}$. Here the trace $\rm{tr}(TU)$ is a volume form on $M$ by $\so3$-equivariance of $T$ and $U$. 
Let 
\begin{align*}
    \cal{P}:=\cal{M}\times \Om^3(P;\sym3)^{\so3}_{hor}.
\end{align*}
Here $\tau\in \cal{G}$ acts on $\cal{P}$ by $(e,U)\cdot \tau=(\tau^*e,\tau^*U)$. Then the horizontal distribution 
\begin{align*}
    H_{(e,U)}\cal{P}:=H_{e}\cal{M}\oplus \Om^3(P;\sym3)^{\so3}_{hor} \quad\text{for all } (e,U)\in \cal{P}
\end{align*}
gives a connection on $j:\cal{P}\to T^*\bf{M}$. On $H\cal{P}\to \cal{P}$, we have a non-degenerate  2-form $\Om' \in \Om^0(\cal{P},\w^2(H\cal{P})^*)$ by 
\begin{align}\no
    \Om'_{(e,U)}((T_1,W_1),(T_2,W_2))=\int_{M}\mathrm{tr}(T_1W_2-T_2W_1)
\end{align}
for $(T_1,W_1), (T_2,W_2)\in H_{(e,U)}\cal{P}$. By Lemma \ref{lem:pre1} stated below, we see $\Om'=j^*\Om$, where $j^*\Om$ is the pull-back of $\Om$ by $j:\cal{P}\to T^*\bf{M}$.

Fix a coordinate neighborhood $(\cal{U};x^1,x^2,x^3)$ of M and a local section $f:\cal{U}\to P$. For each $e \in \mathcal{M}$, we set an $\mathrm{M}(3;\R)$-valued function $C=(C_{ij})$ on $\cal{U}$ by $f^*e^i=C_{ij}dx^j$ for $i=1,2,3.$ Denote by $^tC$ the transvese matrix of $C$. For $\g\in \bf{M}$, put $\rm{vol}(\g)$ = (the volume element of $\g$). Then by definition $\vol (e)=\vol (\g)$ when $\g=i(e)$.
\begin{lemma}\label{lem:pre1} On the neighborhood $(\cal{U};x^1,x^2,x^3)$, the map (\ref{eq:pre3})
\begin{align*}
j:\cal{P}\to  T^*\bf{M},\quad j(e,S\mathrm{vol}(e))=(\g,\pi)
\end{align*}
is given by
\begin{equation*}
\begin{split}
&(\g_{ij})={}^tCC \quad\text{and}\quad (\pi^{ij})=\frac{1}{2}C^{-1}(f^*S)({}^tC)^{-1}\mathrm{vol}(e),
\end{split}
\end{equation*}
where  $\g=\g_{ij}dx^i dx^j$, $\pi=\pi^{ij}\dfrac{\partial}{\partial x^i}\dfrac{\partial}{\partial x^j}$ and $S\in \Om^0(P;\sym3)^{\so3}$.
\end{lemma}
\begin{proof}
By the derivation of $s^*e^i=C_{ij}dx^j$, we see that
\begin{align*}
    i_{*}:H\cal{M} \to T\bf{M}, \quad i_{*}(e,T)=(\g,\rho)
\end{align*}
is given by $(\g_{ij})={}^tC C$ and $(\rho_{ij})=2{}^tC(f^*T) C$. Here $\g=\g_{ij}dx^idx^j$ and $\rho=\rho_{ij}dx^idx^j$ on $(\cal{U};x^1,x^2,x^3)$. Thus by the definition of $j$ (\ref{eq:pre3}),
\begin{align*}
     \langle \pi,\rho\rangle =\int_{M}\mathrm{tr}(TS)\mathrm{vol}(e)=\int_{M}\rm{tr}\left(\frac{1}{2}C^{-1}(f^*S)({}^tC)^{-1}(\r_{ij})\right)\rm{vol}(e)
\end{align*}
 for every $(\g,\r)=i_{*}(e,T)$. Hence $(\pi_{ij})=\frac{1}{2}C^{-1}(f^*S)({}^tC)^{-1}\rm{vol}(e)$.
\end{proof}

\subsection{$\su3$- and $G_2$-structures}
In this subsection we review briefly some results in \cite{Chi}. For general properties of $\su3$- and $G_2$-structures, see for example \cite{BS, CLSS, CS, Joy}.

In general, a {\it G-structure} on an $n$-manifold $N$ is a subbundle of the frame bundle $\mathrm{Fr}(N)$ with structure group contained in $G \subset \gl{n}$.
We look at $\su3$-structures on a 6-manifold $X$, where $\su3 \subset \gl6$ is embedded in the natural way. Then an $\su3$-structure can be identified with a certain pair $(\om, \psi)$ of a 2-form $\om$ and a 3-form $\psi$.
An $\su3$-structure $(\o,\psi)$ is called {\it half-flat} if $d(\o\w\o)=0$ and $d\psi=0$;  and {\it torsion-free} if $d\o=0$ and $d\psi=d\hat{\psi}=0$. Here $\hat{\psi}$ is the dual of $\psi$ \cite{Hit}. If $(\o,\psi)$ is torsion-free, then the associated Riemannian metric on $X$ is Ricci-flat. 

Next we look at $G_2$-structure on a 7-manifold $Y$, where the Lie group $G_2 \subset \gl7$ is embedded as the automorphism group of the cross product on $\R^7$. Then a $G_2$-structure can be identified with a certain 3-form $\phi$.
A $G_2$-structure $\phi$ is called {\it torsion-free} if $d\phi=0$ and $d\star\phi=0$, where $\star$ is the Hodge star naturally defined by $\phi$. If $\phi$ is torsion-free, then the associated Riemannian metric on $Y$ is Ricci-flat.

Let $\psym3$ and $\Om^3(P;\psym3)^{\so3}_{hor}$ be the positive-definite subsets of $\sym3$ and $\Om^3(P;\sym3)^{\so3}_{hor}$ with respect to the orientation of $M$. Put 
\begin{align*}
    \cal{P}^{+}:=\cal{M}\times \Om^3(P;\psym3)^{\so3}_{hor} \subset \cal{P}.
\end{align*}
Let $\cal{A}\subset \Om^1(P;\fso3)^{\so3}$ be the space of connection 1-forms on the  principle bundle $P=M\times \so3$. Fix a basis $\{Y_1,Y_2,Y_3\}$ of $\fso3$ with $[Y_i,Y_j]=\e_{ijk}Y_{k}$ for $i,j=1,2,3$. Here $\e_{ijk}$ is Levi-Civita's symbol. By the basis, we write a connection 1-form $a \in \cal{A}$ as $a=a^iY_i$. 
The following $\so3$-invariant $\su3$-structures were studied in \cite{Chi}. 
\begin{definition}[\cite{Chi}, Definition 3.1]\label{def:pre1}
An $\su3$-structure $(\omega,\psi)$ on $P$ is said to be  {\it $\so3$-invariant (non-integrable) special Lagrangian fibered $\su3$-structure} if it satisfies:
\begin{itemize}
\item $\omega$ and $\psi$ are invariant for the right action of $\so3$ on $P$;
\item restrictions of $\omega$ and $\psi$ to fibers $F_m \subset P$ vanish:  $\omega|_{F_m}=0$ and $\psi|_{F_m}=0$ for all $m \in M$.
\end{itemize}
\end{definition}
Let $(e,S\rm{vol}(e), a) \in \cal{P}^{+}\times \cal{A}$. Here $S \in \Om^0(P;\psym3)^{\so3}$.
The space $\mathcal{P}^{+}\times \cal{A}$ is related to $\su3$-structurs in Definition \ref{def:pre1} as follows:

\begin{prop}[\cite{Chi}, Example 3.3 and Theorem 3.5]\label{prop:pre3}
One-to-one correspondence  between $\su3$-structures $(\o,\psi)$ in Definition \ref{def:pre1} and the triples $(e,S\rm{vol}(e), a)\in \cal{P}^{+}\times \cal{A}$ is given by
\begin{align*}
&\om=\dets^{-\frac{1}{2}}\tilde{S}_{ij}a^i\wedge e^j,\\
&\psi=-(\det S)e^1\wedge e^2 \wedge e^3 + e^1\wedge a^2\wedge a^3+e^2\wedge a^3 \wedge a^1 + e^3\wedge a^1 \wedge a^2, 
\end{align*}
where $\tilde{S}$ is the adjugate matrix of $S$.
\end{prop}
Note that $\tilde{B}=(\tilde{B}_{ij})$, $\tilde{B}_{ij}=\dfrac{1}{2}\e_{i\a\b}\e_{j\g\d}B_{\g\a}B_{\d\b}$ for $B \in \mathrm{M}(3;\R)$, and $B\tilde{B}=\det{B}\cdot I$. Here $I=(\d_{ij})$ is the identity matrix. Denote by $\det{B}$ and $\rm{tr}(B)$ the determinant and trace of $B$. For $S\in \Om^0(P;\sym3)^{\so3}$, define $S_{ij;k}$ by $(d_HS)_{ij}:=(dS+[a,S])_{ij}=S_{ij;\a}e^{\a}$ for  $i,j,k=1,2,3.$ A connection $a$ is called the {\it Levi-Civita connection} of $e$ if $d_He:=de+a\w e=0$, and denoted by $a_{LC}$.

\begin{prop}[\cite{Chi}, Corollary 6.5]
An $\su3$-structure $(\o,\psi)$ in Definition \ref{def:pre1} is half-flat if and only if $(e,S\rm{vol}(e),a)$ determined by Proposition \ref{prop:pre3} satisfies
\begin{align*}
    a=a_{LC}\quad\text{and}\quad S_{i\a;\a}=0 \quad \text{for}\quad i=1,2,3.
\end{align*}
\end{prop}
Let $\cal{P}^+$ embed into $\cal{P}^{+}\times \cal{A}$ by 
\begin{align*}
  \cal{P}^{+}\ni (e,S\rm{vol}(e)) \mapsto (e,S\rm{vol}(e),a_{LC}) \in \cal{P}^{+}\times \cal{A}.
\end{align*}
Take a curve $c(t)=(e(t),S(t)\rm{vol}(e(t)),a_{LC}(t))$ in $\cal{P}^{+}\subset \cal{P}^{+}\times \cal{A}$ defined on an open interval  $(t_1,t_2)$. Then $c(t)$ defines a $G_2$-structure $\phi$ on $P\times (t_1,t_2)$ by
\begin{align}\label{eq:pre8}
 \phi= \dets^{\frac{1}{2}}\om(t)\wedge dt + \psi(t),
 \end{align}
where $(\o(t),\psi(t))$ is the $\su3$-structure deternined by $c(t)$ for each $t$. Note that we set $f=\dets^{\frac{1}{2}}$ of $\phi= f(t)\om(t)\wedge dt + \psi(t)$ in (\cite{Chi}, Subsection 6.2).
Define $G=(G_{ij})$ by $d_{H}a=G_{\a\b}\hat{e}^{\b}Y_{\a}$, where $d_Ha:=a +\dfrac{1}{2}[a\w a]$ and $\hat{e}^i=\dfrac{1}{2}\e_{i\a\b}e^{\a}\w e^{\b}$ for $i=1,2,3$. If $a=a_{LC}$, then $G_{ij}$ coincides with the orthonormal representation of the Einstein tensor of $\g=i(e)$ by local coframes $\{e^1,e^2,e^3\}$. The torsion-free conditions for $\phi$ is described by a constrained dynamical system on $\cal{P}^{+}$ as follows. (Note that $G$ is $\Om$ in \cite{Chi})
\begin{prop}[\cite{Chi}, Theorem 6.11]\label{prop:pre5}
The $G_2$-structure $\phi$ (\ref{eq:pre8}) is torsion-free if and only if $c(t) \in \cal{P}^{+}$ satisfies
\begin{align}
   \label{eq:pre9} &S_{i\a;\a}=0;\\
   \label{eq:pre10} &\pt{e^i}{t}=\tilde{S}_{i\a}e^{\a},\quad \pt{S}{t}=-G-\mathrm{tr}(\tilde{S})S +2\dets I 
\end{align}
for all $t\in (t_1,t_2)$ and for $i=1,2,3$.
\end{prop}

Finally, we remark $\cal{G}$-invariance of equations in Proposition \ref{prop:pre5}. We may restrict the action of $\cal{G}$ on $\cal{P}$ to $\cal{P}^+$. Then by the forms of equations in Proposition \ref{prop:pre5}, we see that if $c(t)=(e(t),S(t)\rm{vol}(e(t)))$ is a solution then so is $c(t)\cdot \tau=(\tau^*e(t),\tau^*S(t)\rm{vol}(\tau^*e(t)))$ for all $\tau\in \cal{G}$.

\section{The constrained Hamiltonian dynamical system}\label{sec:con}
Let $\Om^3(M,S^2_{+}TM)$ be the subspace of positive-definite elements of $\Om^3(M,S^2TM)$ with respect to the orientation of $M$. Put
\begin{align}\no
    T^*\bf{M}^{+}:=\bf{M}\times\Om^3(M,S_{+}^2TM) \subset T^*\bf{M}.
\end{align}
Then the principal $\cal{G}$-bundle $j: \cal{P}\to T^*\bf{M}$ restricts to $j:\cal{P}^+\to T^*\bf{M}^{+}$. 
Moreover the connection on $\cal{P}$ also restricts to $\cal{P}^{+}$. For a curve $\un{c}(t)=(\g(t),\pi(t))$ in $T^*\bf{M}^+$, a curve $c(t)$ in $\cal{P}^{+}$ is called a {\it lift} of $\un{c}(t)$ if it satisfies $j(c(t))=\un{c}(t)$ and $\dot{c}(t)\in H_{c(t)}\cal{P}^+$. Here $\dot{c}(t)$ denotes the derivative of $c(t)$ at $t$, and $H\cal{P}^+$ the horizontal distribution on $\cal{P}^{+}$. 
We can construct a lift $c(t)$ of $\un{c}(t)$ for every curve $\un{c}(t)$ in $T^*\bf{M}$. (For example, by a fixed global coframe on $M$, the Gram-Schmidt process and Lemma \ref{lem:pre1}, we construct a curve $c'(t)$ in $\cal{P}$ with $j(c'(t))=\un{c}(t)$, and determine $\tau(t)\in \cal{G}$ such that $c'(t)\cdot \tau(t)$ is horizontal by solving an ODE on $\cal{G}$. The ODE is a smooth family of ODEs on $\so3$ in form $\dot{h}h^{-1}=A(t)$, where $h \in \so3$ and $A(t) \in \fso3$.)
A lift $c(t)$ gives a $G_2$-structure $\phi$ by (\ref{eq:pre8}). Note that any two lifts $c(t)$ and $c'(t)$ of $\un{c}(t)$ satisfy $c'(t)=c(t)\cdot\tau$ for some $\tau\in \cal{G}$. Then  $\su3$-structures $(\o,\psi)$ and $G_2$-structures $\phi$ corresponding to $c(t)$ and $c'(t)$ are isomorphic by $\tau: P\to P$.

Define a Hamiltonian function $H$ on $(T^*\bf{M},\Om)$  by
\begin{align}\label{eq:con12}
    H(\ga,\pi)= -\frac{1}{2}\int_{M}\mathrm{R}(\ga)\mathrm{vol}(\ga) + 8\int_{M}\det{(\pi^{i}_j)}\mathrm{vol}(\ga)^{-2} \quad\text{for}\quad (\ga,\pi)\in T^*\mathbf{M}.
\end{align}
Here $\rm{R}(\g)$ denotes the scalar curvature of $\g\in \bf{M}.$ We formulate our result by $H$ , which is different from $H_{G_2}$ mentioned in Introduction. This simplifies the proof of our main result in Section \ref{sec:pro}. The dynamical systems of $H$ and $H_{G_2}$ are mutually transformed  by scaling (See Remark \ref{rem:prop1}).
Let $\nabla$ be the Levi-Civita connection for $\g$.
The main results in the paper are as follows.

\begin{theorem}\label{thm:con1}
Let $\un{c}(t)=(\g(t),\pi(t))$ be a curve in $T^*\bf{M}^+$ defined on $(t_1,t_2)$. Then a lift $c(t)$ of $\un{c}(t)$ gives a torsion-free $G_2$-structure $\phi$ by (\ref{eq:pre8}) if and only if $\un{c}(t)$ satisfies the following:
\begin{itemize}
\item $\un{c}(t)$ is a solution of the Hamiltonian dynamical system of $H$.
\item $\nabla_{\a}\pi^{i\a}(t)=0$ holds for $i=1,2,3$ and for all $t\in (t_1,t_2)$
\end{itemize}
\end{theorem}

\begin{prop}\label{prop:con2}
Let $\un{c}(t)$ be an orbit of the Hamiltonian dynamical system of $H$. Assume that  $\nabla_{j}\pi^{ij}(t_0)=0$ holds for $i=1,2,3$ at some $t_0$. Then $\nabla_{j}\pi^{ij}(t)=0$ holds for $i=1,2,3$ on the whole orbit.
\end{prop}

\begin{remark}\label{rem:con1}
Let $\un{c}(t)$ be a curve in $T^*\bf{M}$ satisfying the conditions in Theorem \ref{thm:con1}, defined on $(t_1,t_2)$. If  $\pi(t)$ is not positive-definite and $\det{(\pi(t))} \neq 0$ for all $t$, then a $G^*_2$-structure is defined on $P\times (t_1,t_2)$ from a lift $c(t)$ of $\un{c}(t)$ by (\ref{eq:pre8}). We conjecture that this $G^*_2$-structure is also torsion-free.  In general, a torsion-free $G_2^*$-structure on a 7-manifold $Y$ defines a Ricci-flat pseudo-Riemannian metric of signature $(3,4)$ on $Y$. See  (\cite{CLSS}, p.\ 122 and Theorem 2.3.) for more details.
\end{remark}

\section{Variation formulas}\label{sec:var}
In order to prove Theorem \ref{thm:con1} and \ref{prop:con2}, we prepare some formulas. Identify $\Om^p(P;\R^3)$ with $\Om^p(P;\fso3)$ by the basis $\{Y_1,Y_2,Y_3\}$ of $\fso3$. For given $e\in \cal{M}$, $a\in \cal{A}$ and $A=(A_{ij})\in \Om^0(P;\mathrm{M}(3;\R))^{\so3}$, define $A_{ij;k}$ by $(d_HA)_{ij}:=(dA+[a,A])_{ij}=A_{ij;\a}e^{\a}$ for $i,j=1,2,3.$ Denote by $\dot{\s}$ the derivative $\dfrac{\partial \s}{\partial t}$ of a differential form $\s$.
Moreover put $\hat{e}^i=\dfrac{1}{2}\e_{i\a\b}e^{\a}\w e^{\b}$ for $i=1,2,3.$
\begin{lemma}\label{lem:var1}Let $e(t)$ be a curve in $\cal{M}$. Suppose that 
$\dpt{e^i}{t}=P_{i\a}e^{\a}$ and $P_{ij}=P_{ji}$ for $i,j=1,2,3.$
Then 
$\dfrac{\partial a_{LC}^i}{\partial t}=-\e_{i\a\b}P_{\g\a;\b}e^{\g}$ for $i=1,2,3.$
\end{lemma}

\begin{proof}
Put $\dot{a}_{LC}^i=Q_{ij}e^j$ for $i=1,2,3$. Then
\begin{align*}
  0&=  \pt{(d_He)}{t} ={} \pt{}{t}(de+[a_{LC}\w e])=d\dot{e}+[\dot{a}_{LC}\w e]+[a_{LC}\w \dot{e}]
    ={} d_H\dot{e}+[\dot{a}_{LC}\w e]\\
    &={} P_{ij;k}e^{kj}Y_i+[Q_{\a k}e^kY_{\a}\w e^{\b}Y_{\b}] 
    ={} (\e_{\g kj}P_{ij;k}+\e_{\g k\b}\e_{i\a\b}Q_{\a k})\hat{e}^{\g}Y_i
\end{align*}
Thus $\e_{\g kj}P_{ij;k}+\e_{\g k\b}\e_{i\a\b}Q_{\a k}=0$ for $\g,i=1,2,3$. Since $(P_{ij})$ is symmetric, we see $Q_{ij}=-\e_{i\a\b}P_{j\a;\b}$ for $i,j=1,2,3.$
\end{proof}

\begin{lemma}\label{lem:var2}
Let $(e,a,A)$ be a curve in $\cal{M}\times\cal{A}\times\Om^0(P;\rm{M}(3;\R))^{\so3}$. Suppose that 
    $\dpt{e^i}{t}=P_{ij}e^j$, $\dpt{a^i}{t}=Q_{ij}e^j$ and $\dpt{A}{t}=B$ for $i=1,2,3.$ Then 
    \begin{align*}
    \pt{}{t}(A_{ij;k})=B_{ij;k}-A_{ij;\a}P_{\a k}+\e_{i\a\g}A_{\g j}Q_{\a k}+\e_{j\a\g}A_{i\g}Q_{\a k}\quad\text{for}\quad i,j,k=1,2,3.
    \end{align*}
In particular,  
\begin{align*}
\dpt{}{t}(A_{i\a;\a})=B_{i\a;\a}-A_{i\a;\b}P_{\b\a}+\e_{i\b\g}A_{\g\a}Q_{\b\a}+\e_{\a\b\g}A_{i\g}Q_{\b\a}\quad\text{for}\quad i=1,2,3.
\end{align*}
\end{lemma}
\begin{proof}
By definition, we have 
$\dpt{}{t}(d_HA)=d_HB+(Q_{kj}e^j)[Y_k,A]$ and $\dpt{}{t}(d_HA)_{ij}=\dpt{}{t}(A_{ij;k})e^k+A_{ij;k}(P_{k\a}e^{\a})$ for $i,j=1,2,3$.
By comparing these, we obtain the equations in Lemma \ref{lem:var2}.
\end{proof}

For $e\in\cal{M}$, denote by $R(e)$ the scalar curvature of $\g=i(e)\in \bf{M}$. 
\begin{lemma}\label{lem:var3}
Let $(e(t),S(t))$ be a curve in $\cal{M}\times \Om^0(P;\sym3))^{\so3}$. Put $\dfrac{\partial S}{\partial t}=\dot{S}$ and $\dfrac{\partial e^i}{\partial t}=P_{ij}e^j$ for $i=1,2,3.$ Then 
\begin{align*}
   &\dpt{}{t}\int_{M}\rm{R}(e)\rm{vol}(e)=-2\int_{M}\rm{tr}(GP)\rm{vol}(e),\quad
   \dpt{}{t}(\rm{vol}(e))=\rm{tr}(T)\rm{vol}(e),\\
   &\dpt{}{t}(\det{S})=\rm{tr}(\tilde{S}\dot{S}).
   \end{align*}
\end{lemma}
\begin{proof}
The last two are straightforward. Let $\langle A,B \rangle:= -\dfrac{1}{2}\rm{tr}(AB)$ for $A,B\in \fso3$. Since $G$ is the Einstein tensor for $\g=i(e)$ on a 3-manifold $M$, we have $\rm{tr}(G)=-\dfrac{1}{2}R(e).$ Thus 
\begin{align*}
    \pt{}{t}\int_{M}\rm{R}(e)\rm{vol}(e) &=-2\pt{}{t}\int_{M}\rm{tr}(G)\rm{vol}(e)
    = -2\int_{M}\pt{}{t}\langle d_Ha, e\rangle \\
    &= -2\int_{M}\left(\langle d_H\dot{a},e\rangle+\langle G_{ij}Y_i\hat{e}^j,\dot{e}\rangle\right) =-2\int_{M}\rm{tr}(GP)\rm{vol}(e).
\end{align*}\end{proof}
\section{Proof of the main results}\label{sec:pro}
\subsection{Proof of Theorem \ref{thm:con1}}
Let $\un{c}(t)=(\g(t),\pi(t))$ be a curve in $T^*\bf{M}$ and $c(t) \subset \cal{P}$  a lift curve of $\un{c}(t)$. We can easily see that $\pi(t)$ satisfies positive-definiteness and $\nabla_{\a}\pi^{i\a}=0$ for $i=1,2,3$ if and only if $S(t)$ of the lift $c(t)=(e(t),S(t)\rm{vol}(e(t)))$ satisfies positive-definiteness and $S_{i\a;\a}=0$ for $i=1,2,3.$ Here note that by Lemma \ref{lem:pre1}, we have $\pi^{ij}\dpt{}{x^i}\dpt{}{x^j}=\dfrac{1}{2}(f^*S)_{ij}e_ie_j$ for any local section $f:(\cal{U};x^1,x^2,x^3)\to P(=M\times \so3)$, where $(e_1,e_2,e_3)$ is the dual orthonormal basis of $(f^*e^1,f^*e^2,f^*e^3)$.

Since the 2-form $\Om'$ on $H\cal{P}$ is non-degenerate, for any function $F'$ on $\cal{P}$, we can define uniquely the horizontal vector field $X^{F'} \in \Om^0(\cal{P},H\cal{P})$ by $dF'(Z)=\Om'(X^{F'},Z)$ for all $Z\in H\cal{P}$. Also for a function $F$ on $(T^*\bf{M},\Om)$, denote by $X^F$ the Hamiltonian vector field of $F$. From now, for a function $F$ on $T^*\bf{M}$, put $F'=j^*F$, where $j:\cal{P}\to T^*\bf{M}$. Then by $j^*\Om=\Om'$, the vector field $X^{F'}$ is the lift of $X^F$ satisfying $j_{*}X^{F'}=X^{F}$. Thus it suffices to check $X^{H'}=X^s$, where $X^s$ is the horizontal vector field on $\cal{P}$ defined by solution curves $c(t)=(e(t),S(t)\rm{vol}(e(t))) \subset \cal{P}$ of (\ref{eq:pre10}) in Proposition \ref{prop:pre5}.

For a horizontal curve $c(t)=(e(t),S(t)\rm{vol}(e(t))) \subset \cal{P}$, set $T(t)$
 and $V(t)$ by\\ $\dpt{c(t)}{t}=(T(t),V(t)\rm{vol}(e(t))) \in H_{c(t)}\cal{P}=\Om^0(P,\sym3)^{\so3}\oplus \Om^3(P;\sym3)^{\so3}_{hor}$. Then applying Leibniz rule gives 
\begin{equation}\label{eq:pro13}
\begin{split}
    &\pt{e^i}{t}=T_{ij}e^j\quad \text{for}\quad i=1,2,3, \\
    &\pt{S}{t}=V-\rm{tr}(T)S.
\end{split}
\end{equation}

Define $T^s,V^s\in \Om^0(P,\sym3)^{\so3}$ by $X^s=(T^s,V^s\rm{vol}(e)) \in H_{(e,S\rm{vol}(e))}$ for each $(e,S\vol (e)) \in \cal{P}$. (\ref{eq:pre10}) and (\ref{eq:pro13}) give 
\begin{equation}\label{eq:pro14}
    \begin{split}
        &T^s=\tilde{S},\\
        &V^s= -G+2\dets \cdot I.
    \end{split}
\end{equation}

Put 
\begin{equation}\label{eq:pro15}
    \begin{split}
        &H_1(\g,\pi)=\int_{M}\rm{R}(\g)\vol (\g), \\
        &H_2(\g,\pi)=8\int_{M}\det{(\pi)}\vol (\g)^{-2}
    \end{split}
\end{equation}
for $(\g,\pi) \in T^*\bf{M}$. Then $H=-\dfrac{1}{2}H_1+H_2.$ Lemma \ref{lem:pre1} gives 
\begin{equation}\no
    \begin{split}
        &H_1'(e,S\vol (e))=\int_M\rm{R}(e)\vol (e), \\
        &H_2'(e,S\vol (e))=\int_M\det (S) \vol (e) 
    \end{split}
\end{equation}
for $(e,S\vol (e))\in \cal{P}$. Here $\det (\pi)=\dfrac{1}{8}\det (S)\rm{vol}(e)^3$ for $j(e,S\vol (e))=(\g,\pi) \in T^*\bf{M}$.
Take any $Z=(T,V\vol (e))\in H_{(e,S\vol (e))}\cal{P}$.
Lemma \ref{lem:var3} gives 
\begin{equation}\label{eq:pro17}
    \begin{split}
        &dH_1'(Z)=-2\int_M\tr (GT)\vol (e), \\
        &dH_2'(Z)= \int_M\left(\tr (\tilde{S}V)-2\dets \tr (I\cdot T)\right)\vol (e).
    \end{split}
\end{equation}
On the other hand, we have 
\begin{align}\label{eq:pro18}
    \Om'(X^{H_i'},Z)=\int_M\tr (T^{H_i'}V-TV^{H_i'})\vol (e)
\end{align}
for $X^{H_i'}=(T^{H_i'},V^{H_i'}\vol (e)) \in H_{(e,S\vol (e))}\cal{P}$. Comparing (\ref{eq:pro17}) with (\ref{eq:pro18}) gives 
\begin{equation}\label{eq:pro19}
    \begin{split}
        &T^{H_1'}=0,\quad V^{H_1'}=2G, \\
        &T^{H'_{2}}=\tilde{S},\quad V^{H_2'}=2\dets \cdot I.
    \end{split}
\end{equation}
Thus $X^{H'}=-\dfrac{1}{2}X^{H_1'}+X^{H_2'}=X^s$. \qed 

\subsection{Proof of Proposition \ref{prop:con2}}
Let $c(t)\subset \cal{P}$ be a lift curve of an orbit $\un{c}(t)$ of the Hamiltonian dynamical system of $H$. Then by Theorem \ref{thm:con1}, $c(t)$ is a solution of (\ref{eq:pre10}) in Proposition \ref{prop:pre5}. It suffices to prove that if the solution $c(t)$ satisfies (\ref{eq:pre9}) at some $t_0$, then $c(t)$ satisfies (\ref{eq:pre9}) at all $t$.
\quad\\

By (\ref{eq:pre10}), Lemma \ref{lem:var1} and \ref{lem:var2}, we have 
\begin{align*}
    \pt{}{t}(S_{i\a;\a}) &= (-\rm{tr}(\tilde{S})S-G+2\dets I)_{i\a;\a}- S_{i\a;\b}\tilde{S}_{\b\a}\\
    \no&\quad + \e_{i\b\g}S_{\g\a}(-\e_{\b kl}\tilde{S}_{\a k;l})+\e_{\a\b\g}S_{i\g}(-\e_{\b kl}\tilde{S}_{\a k;l}) \\
   \no &= -\tilde{S}_{\b\b}S_{i\a;\a}-S_{\a\g;\g}\tilde{S}_{\a i}+S_{\g\a;i}\tilde{S}_{\a\g}-\dets_{;i}
\end{align*}
for $i=1,2,3, $where $d\dets =\dets_{;k}e^k$. Moreover we have 
$\tilde{S}_{\a\b}S_{\a\b;i} = \tr (\tilde{S}S_{;i})=\dets_{;i}$ for $i=1,2,3$. Hence we obtain 
\begin{align}\label{eq:pro20}
    \pt{}{t}(S_{i\a;\a})=-\tilde{S}_{\a\a}S_{i\b;\b}-\tilde{S}_{i\a}S_{\a\b;\b} \quad\text{for}\quad i=1,2,3.
\end{align}
Here (\ref{eq:pro20}) is a first-order homogeneous linear ODE for  $(S_{1\a;\a}(t),S_{2\a;\a}(t),S_{3\a;\a}(t))$.  Thus by the uniqueness of the solution,  we can conclude that if the solution $c(t)$ satisfies (\ref{eq:pre9}) at some $t_0$, then $c(t)$ satisfies (\ref{eq:pre9}) at all $t$. \qed

\section{Properties of the Hamiltonian dynamical system}\label{sec:prop}
\subsection{Scaling of the dynamical system}
Let $c(t)=(e(t),S(t)\vol (e(t)))$ be a horizontal curve in $\cal{P}$. Fix $\a,\b,\k >0$. Define $c'(t')=(e'(t'),S'(t')\vol (e'(t'))) \subset \cal{P}$ by $e'(t')=\a e(\k t')$ and $S'(t')=\b S(\k t')$. Recall that $H, H_1$ and $H_2$ defined by (\ref{eq:con12}) and (\ref{eq:pro15}).
\begin{prop}\label{prop:prop1}
Take $a,b >0$ with $\a=2^{\frac{1}{2}}a^{\frac{1}{2}}b^{\frac{1}{4}}\k^{-\frac{1}{4}}$ and $\b=b^{-\frac{1}{2}}\k^{\frac{1}{2}}$. Then $c(t)$ is a lift of an orbit of the Hamiltonian dynamical system of $H$ if and only if $c'(t')$ is a lift of that of $-aH_1+bH_2$. 
\end{prop}
\begin{proof}
Assume that $c(t)$ is a lift of an orbit of the Hamiltonian dynamical ststem of $H$. 

Put $\dpt{c(t)}{t}=(T(t),V(t)\vol (e(t)))$ and $\dpt{c'(t')}{t'}=(T'(t'),V'(t')\vol (e'(t')))$. Then by (\ref{eq:pro13}), we have 
\begin{align*}
    &\pt{(e')^i}{t'}=\k T(\k t')_{ij}(e')^j\quad \text{for}\quad i=1,2,3, \\
    &\pt{S'}{t'}= \b\k \pt{S}{t}(\k t')= \b\k V(\k t')-\tr (T'(t'))S'(t').
\end{align*}
Thus $T'(t')=\k T(\k t')$ and $V'(t')=\b\k V(\k t')$. Moreover by assumption and (\ref{eq:pro14})
\begin{align*}
    &T(\k t')=\tilde{S}(\k t') , \\
    &V(\k t')= -G(\k t')+ 2\det{(S(\k t'))}\cdot I.
\end{align*}
Using $\tilde{S}(\k t')=\b^{-2}\tilde{S'(t')}$,  $\det(S(\k t'))=\b^{-3}\det (S'(t'))$ and $G(\k t')=\a^2G'(t')$, where $G'(t')$ denotes the orthonormal representation of the  Einstein tensor defined by $e'(t')$, we have
\begin{equation}\label{eq:prop21}
    \begin{split}
        &T'(t')=\k\b^{-2}\tilde{S'}(t'), \\
        &V'(t')=-\k\a^2\b G'(t') +2\k\b^{-2} \det (S'(t'))\cdot I.
    \end{split}
\end{equation}
On the other hand, by (\ref{eq:pro19}), we have 
\begin{equation}\label{eq:prop22}
    \begin{split}
        &-aT^{H_1'}+bT^{H_2'}=b\tilde{S'}, \\
        &-aV^{H_1'}+bV^{H_2'}=-2aG'+2b\det{(S')}\cdot I.
    \end{split}
\end{equation}
By assumption, (\ref{eq:prop21}) coincides with (\ref{eq:prop22}). Hence $c'(t')$ is a lift of $-aH_1+bH_2$. The inverse is the same as above.
\end{proof}

\begin{remark}\label{rem:prop1}
Let $\un{c}(t)$ be a curve in $T^*\bf{M}$, and $c(t)=(e(t),S(t)\vol (e(t)))$ a lift curve of $\un{c}(t)$ in $\cal{P}$. In order to apply Proposition \ref{prop:prop1} to $H_{G_2}=-H_1+\dfrac{1}{8}H_2$, take constants $\k>0$, $\a=2^{-\frac{1}{4}}\k^{-\frac{1}{4}}$ and $\b=2^{\frac{3}{4}}\k^{\frac{1}{2}}$. Then by Proposition \ref{prop:prop1}, we see that $\un{c}(t)$ is an orbit of the Hamiltonian dynamical system of $H_{G_2}$ if and only if $c'(t'):=(\a^{-1}e(\k^{-1}t'),\b^{-1}S(\k^{-1}t')\vol (\a^{-1}e(\k^{-1}t')))$ is a lift  of an orbit of the Hamiltonian dynamical system of $H$ defined by (\ref{eq:con12}). Thus by Theorem \ref{thm:con1}, we see that $\un{c}(t)$ satisfies the conditions in Theorem \ref{a} if and only if $c'(t')$ yields a torsion-free $G_2$-structure by (\ref{eq:pre8}). This is the statement of Theorem \ref{a}. Here note that the positive-definiteness and the divergence-free conditions are preserved by the scaling with positive constants. Thus we can also deduce Proposition \ref{b} form Proposition \ref{prop:con2}.
\end{remark}

\subsection{Variations of some functionals} 

Let $\un{c}(t)=(\g(t),\pi(t))$ be an orbit of the Hamiltonian dynamical system of $H$, and $c(t)=(e(t),S(t)\vol (e(t)))$ be a lift of $\un{c}(t)$ in $\cal{P}$. Put 
$H_3(\g,\pi)=\int_{M}\vol (\g)$ for $(\g,\pi)\in T^*\bf{M}.$
By (\ref{eq:pro14}) and (\ref{eq:pro17}), we can easily deduce variation formulas of $H_1$, $H_2$ and $H_3$ along the orbit $\un{c}(t)$:
\begin{align*}
    \pt{H_1}{t}&=\pt{}{t}\int_M\rm{R}(\g)\vol (\g) =-2\int_M\tr{(G\tilde{S})}\vol (e)=\int_M\left(\rm{R}(e)\tr (\tilde{S})-2\tr (\ric\tilde{S})\right)\vol (e), \\
    \pt{H_2}{t}&= 8\pt{}{t}\int_M\det (\pi)\vol (\g)^{-2}=\frac{1}{2}\pt{H_1}{t},\\
    \pt{H_3}{t}&= \pt{}{t}\int_M\vol (\g) =\int_M\tr (\tilde{S}) \vol (e).
\end{align*}
Here $\ric$ denotes the  Ricci tensor $\ric =G-\tr{(G)}I$. For example, by the third formula, we can see that if $\pi(t_0)$ is positive- or negative-definite at every point on $M$, then
\begin{align*}
\left.\pt{H_3}{t}=\frac{\partial}{\partial t}\int_{M}\mathrm{vol}(\g(t))\right|_{t=t_0} >0.
\end{align*}
 Here note that  $\tilde{S}=\det {(S)}S^{-1}$ and $\tr{(\tilde{S})}=\det (S)\tr (S^{-1})$.

\subsection{Observation on Bryant-Salamon's examples}

Take $\s\in \R$ and $e_0\in \cal{M}$ satisfying $G=-\s\cdot I$. Here we identify $I$ with the identity matrix-valued function in $\Om^0(P,\sym3)^{\so3}$. Since $M$ is a 3-manifold, $\tr (G)=-\dfrac{1}{2}\rm{R}(e)$. Then $\g=i(e_0)\in \bf{M}$ is a constant curvature Riemannian metric with scalar curvature $6\s$.
Define an enbedding $i_{e_0}: \R_{>0}\times \R \to \cal{P}$ by 
\begin{align*}
    \R_{>0}\times \R \ni (a,b) \mapsto (ae_0,bI \vol (ae_0)) \in \cal{P}.
\end{align*}
Then by Proposition \ref{prop:pre5} and Theorem \ref{thm:con1}, we can easily see that the integral curves of the horizontal vector fields $X^{H'}$ on the image $i_{e_0}(\R_{>0}\times \R)$ are restricted in $i_{e_0}(\R_{>0}\times \R)$ and are solutions of the following dynamical system:
\begin{equation}\label{eq:propdyn1}
    \begin{split}
        &\frac{da}{dt}=ab^2, \\
        &\frac{db}{dt}=\s a^{-2}-b^3\quad \text{for}\quad (a,b)\in\R_{>0}\times \R.
    \end{split}
\end{equation}
Here we remark about Bryant-Salamon's examples in (\cite{BS}, Section 3). The examples are diffeomorphic to $M\times S^3\times (t_1,t_2)$. The torsion-free $G_2$-structures of the examples are constructed by one-parameter families of $\su3$-structures on $M\times S^3$ satisfying the conditions in Definition \ref{def:pre1}. Moreover the $\su3$-structures are constructed by constant curvature metrics on $M$ and on $S^3$. Thus the examples are locally isomorphic to torsion-free $G_2$-structures constructed by (\ref{eq:pre8}) form solution curves $(a(t)e_0,b(t)I\vol (a(t)e_0))$ of (\ref{eq:propdyn1}) satisfying $b(t)>0$.

Put $x(t)=a(t)^2$ and $y(t)=a(t)b(t)$. Then (\ref{eq:propdyn1}) is equivalent to the following dynamical system:
\begin{equation}\label{eq:propdyn2}
    \begin{split}
        &\frac{dx}{dt}=2y^2, \\
        &\frac{dy}{dt}=\s x^{-\frac{1}{2}}\quad\text{for}\quad (x,y)\in \R_{>0}\times \R.
    \end{split}
\end{equation}
We can draw the vector field $(\dfrac{dx}{dt},\dfrac{dy}{dt})$ on $\R_{>0}\times \R$. Observing the vector field, it seems that the integral curve $(x(t),y(t))$  of (\ref{eq:propdyn2}) could extend to $[0,+\infty)$ for every initial value $(x(0),y(0))=(x_0,y_0) \in \R_{>0}\times \R$. 
Assume $\s >0.$ Then every integral curve $(x(t),y(t))$ defined on $[0,+\infty)$ has $y(t)>0$ for all $t$ more than sufficiently large $t_1$. Hence the curve $(a(t)e_0,b(t)I\vol (a(t)e_0)) \subset \cal{P}$ corresponding to $(x(t),y(t))$ yields torsion-free $G_2$-structure on $M\times \so3\times (t_1,+\infty)$ by (\ref{eq:pre8}). Next, assume $\s<0$. Then every integral curve $(x(t),y(t))$ defined on $[0,+\infty)$ has $y(t)<0$ for all sufficiently large $t$. Hence by Remark \ref{rem:con1}, the curve in $\cal{P}$ corresponding to $(x(t),y(t))$ yields indefinite $G^*_2$-structure on $M\times \so3\times (t_1,+\infty)$ by (\ref{eq:pre8}). Finally, assume $\s=0$. Then every integral curve $(x(t),y(t))$ satisfy $y(t)=y_0$ for all $t\in [0,+\infty).$ Hence the curve in $\cal{P}$ corresponding to $(x(t),y(t))$ is always contained in $\cal{P}^+$, or otherwise never contained. 

For example, let us consider the case of $\s>0$. One can guess the behavior of integral curves as follows.
If we start from an initial value $(x_0,y_0)$ with $y_0<0$, the smooth integral curve $(x(t),y(t))$ defined on $[0,+\infty)$ yields  indefinite $G^*_2$-structure on $M\times\so3\times (0,t_1)$, and on the other hand yields torsion-free positive-definite $G_2$-structure on $M\times \so3\times (t_1,+\infty)$ for some $t_1>0$. The $\su3$-structures defined by $(x(t),y(t))$ break down at $t_1$. In this case, $G^*_2$- and $G_2$-structures are smoothly continued as an orbit of the Hamiltonian dynamical system of $H$ on $T^*\bf{M}$.

\section*{Acknowledgements}
The author thanks N. Kawazumi for careful reading of the manuscript and for giving many comments. He also thanks T. Ishibashi for his question on the existence of the Hamiltonian or Lagrangian formulation of the dynamical system in Proposition \ref{prop:pre5}. Further he thanks K. Krasnov for many comments and for telling him related works \cite{HK,HKSS}. Finally he thanks the anonymous referee for many helpful and constructive comments. This work was supported by the Leading Graduate Course for Frontiers of Mathematical Sciences and Physics.

\bibliographystyle{alpha}
\bibliography{paper2}
\end{document}